\documentclass{article}

\usepackage{amsmath}
\usepackage{amsfonts}
\usepackage{amsthm}
\usepackage{amssymb}
\usepackage{MnSymbol}
\usepackage{bbm}
\usepackage{fancyhdr}
\usepackage{algorithmic}
\usepackage{listings}
\usepackage{enumerate}
\usepackage{cite}

\newcommand{\R}{\mathbb{R}}

\newcommand{\E}{\mathbb{E}}

\newcommand{\1}[1]{\mathbbm{1}_{\{#1\}}}

\DeclareMathOperator{\Cov}{Cov}

\newtheorem{theorem}{Theorem}[section]

\newtheorem{proposition}[theorem]{Proposition}

\newtheorem{definition}[theorem]{Definition}

\newtheoremstyle{example}{\topsep}{\topsep}%
     {}%         Body font
     {}%         Indent amount (empty = no indent, \parindent = para indent)
     {\bfseries}% Thm head font
     {}%        Punctuation after thm head
     {\newline}%     Space after thm head (\newline = linebreak)
     {\thmname{#1}\thmnumber{ #2}\thmnote{ #3}}%         Thm head spec

\theoremstyle{example}
\newtheorem{example}[theorem]{Example}
\author{Joe Neeman\footnote{Department of Statistics, U.C.\ Berkeley. joeneeman@gmail.com}}
\title{A law of large numbers for weighted plurality}

\begin{document}
\maketitle

\begin{abstract}
Consider an election between $k$ candidates in which
each voter votes randomly (but not necessarily independently)
and suppose that there is a single candidate that every voter
prefers (in the sense that each voter is more likely to
vote for this special candidate than any other candidate).
Suppose we have a voting rule that takes all of the votes
and produces a single outcome and suppose that each
individual voter has little effect on the outcome
of the voting rule.
If the voting rule is a weighted plurality,
then we show that with high probability, the preferred
candidate will win the election. Conversely, we show
that this statement fails for all other reasonable voting rules.

This result is an extension of one by
H\"aggstr\"om, Kalai and Mossel, who
proved the above in the case $k=2$.
\end{abstract}

\section{Introduction}

For elections between two candidates, it is well known
that voting rules in which every voter has a small effect
are good rules in the sense that they ``aggregate
information well:'' if every voter has a small
bias towards the same candidate then that candidate
will win with overwhelming probability. When voters
vote independently, this fact was noted by Margulis~\cite{Marg}
and Russo~\cite{Russo},
whose results were later strengthened by Kahn, Kalai
and Linial~\cite{KKL} and by Talagrand~\cite{Tal}.

When the voters are not independent, the situation is
more complicated. It is no longer true, then, that every
reasonable voting rule aggregates well. In fact,~\cite{HKM}
show that if we want the aggregation to hold for every
distribution of the voters, then weighted majority
functions are the only option. We extend their result
to the non-binary case.

The author would like to thank Elchanan Mossel for suggesting
this problem and providing fruitful discussions.

\section{Definitions and results}
In the introduction, we made a few allusions to ``reasonable''
voting rules. Let us now say precisely what
that means: we will require that our voting rules
do not have a built-in preference for any alternative.
This is a common assumption, and its definition is standard (see, eg.~\cite{BramsFishburn:02}).
In what follows, the notation $[k]$ stands for the set $\{0, \dots, k-1\}$.

\begin{definition}
A function $f: [k]^n \to [k]$ is \emph{neutral} if
$f(\sigma(x)) = \sigma(f(x))$ for all $x \in [k]^n$
and all permutations $\sigma$ on $[k]$, where
$\sigma(x)_i = \sigma(x_i)$.
\end{definition}

Note that in the case $k=2$, a function is neutral
if, and only if, it is anti-symmetric according
to the definition in~\cite{HKM}.

\begin{example}\label{ex:neutral}
  When $k = 2$ and $n$ is odd, then the simple majority
  function (for which $f(x) = 1$ if $\#\{i: x_i = 1\} > \#\{i: x_i = 0\}$)
  is neutral.
  On the other hand, if $n$ is even then in order to fully specify
  the simple majority function, we need to say what happens in the
  case of a tie; the choice of tie-breaking rule will determine
  whether the resulting function is neutral. For example,
  if we define $f(x) = x_1$ for every tied configuration $x$, then
  $f$ is neutral. On the other hand, if $f(x) = 1$ for every tied
  configuration $x$, then $f$ is not neutral.

  The example can be extended to $k \ge 3$. In this case,
  consider
  the tie-breaking rule $f(x) = x_i$ where $i$ is the smallest possible
  number for which $x_i$ is equal to one of the tied alternatives.
  This tie-breaking rule is neutral, and it
  is more natural than setting $f(x) = x_1$
  because it guarantees that the output of $f$ is one of the tied
  alternatives.
\end{example}

\subsection{Weighted plurality functions}
Let us say precisely what we mean by a weighted plurality
function. The definition that we take here generalizes
the definition from~\cite{HKM} of a weighted majority function.

\begin{definition}
\label{def:wp}
A function $f: [k]^n \to [k]$ is a \emph{weighted plurality function}
if there exist weights $w_1, \dots, w_n \in \R_{\ge 0}$ such that
$\sum_i w_i = 1$ and
for all $a, b \in [k]$,
$f(x) = a$ implies that
\[
\sum_{i: x_i = a} w_i \ge \sum_{i: x_i = b} w_i.
\]
\end{definition}

Note that the above definition does not prescribe a particular
behavior if a tie occurs between two alternatives.
If the weights are chosen so that ties never occur, then the
weighted plurality function is clearly neutral.
Moreover, for any set of weights
we can construct a neutral weighted plurality function
with those weights by following the tie-breaking rule
outlined in Example~\ref{ex:neutral}.

\subsection{The influence of a voter}
The final notion that we need before stating our result
is a way to quantify the power of a single voter.
When $k=2$, the notion of \emph{effect} is well-established
and can be found, for example, in~\cite{HKM}.
However, there does not seem to be a well-established way of quantifying
the effect of voters for non-binary social choice functions. Here,
we propose a definition that closely resembles the one
used in~\cite{HKM} for binary functions.

\begin{definition}
Let $f$ be a function $[k]^n \to [k]$ and fix a probability
distribution $P$ on $[k]^n$. The \emph{effect of voter $i$}
is
\[
e_i(f, P) = \sum_{j=1}^k P(f(X) = j | X_i = j) - P(f(X) = j | X_i \ne j),
\]
where $X$ is a random variable distributed according to $P$.
\end{definition}

Note that for the case $k = 2$, the preceding definition reduces to
\[
e_i(f, P) = %P(f(X) =
%1 | X_i = 1) - P(f(X) = 0 | X_i = 0)
%- (P(f(X) = 1 | X_i = 0) - P(f(X) = 0 | X_i = 0))
2(P(f(X) = 1 | X_i = 1) - P(f(X) = 1 | X_i = 0)),
\]
which is just twice the definition in~\cite{HKM} of a voter's effect.
Also, the effect is closely related to the correlation between
the voters and the outcome:
\begin{align*}
P(f(X) = j | X_i = j) - P(f(X) = j | X_i \ne j)
&= \frac{\Cov(\1{f=j}, \1{X_i = j})}{P(X_i = j) P(X_i \ne j)} \\
&\ge 4 \Cov(\1{f=j}, \1{X_i = j})
\end{align*}
and so
\[
e_i(f, P) \ge 4 \sum_j \Cov(\1{f=j}, \1{X_i=j}).
\]

\begin{example}
  The simplest example of $e_i(f, P)$ is when $P$
  is a product measure (ie.\ the $X_i$ are independent)
  and the function $f$
  does not depend on its $i$th coordinate; in that case,
  $P(f(X) = j | X_i = j) = P(f(X) = j | X_i \ne j)$
  for all $j$ and so $e_i(f, P) = 0$. On the other hand, if
  $P$ is a distribution such that $X_1 = X_2 = \cdots = X_n$
  with probability 1, and if $f$ is a plurality function, then
  $P(f(X) = j | X_i = j) = 1$ for all $j$, while
  $P(f(X) = j | X_i \ne j) = 0$; hence, $e_i(f, P) = 1$ for all $i$.

  For a less trivial example, suppose that the $X_i$ are independent
  and uniformly distributed on $[k]$. Let $f$ be an unweighted plurality
  function. Then the Central Limit Theorem implies that $e_i(f, P)
  = O(\frac{1}{\sqrt n})$ as $n \to \infty$.

  On the other hand, suppose that $f$ is still an unweighted
  plurality function and the $X_i$ are independent, but now
  $P(X_i = 1) > P(X_i = j) + \delta$ for some $\delta > 0$
  and all $j \ne 1$. Then Hoeffding's inequality implies
  that $P(f(X) = 1 | X_i) \ge 1 - 2\exp(-\delta^2 n / 4)$ for
  sufficiently large $n$, regardless of the value of $X_i$. In particular,
  this implies that $e_i(f, P) = O(\exp(-\delta^2 n / 4))$. 
  Compared to the case where the $X_i$ are uniformly distributed,
  this demonstrates that $e_i(f, P)$ can depend strongly on $P$,
  even when $P$ is restricted to being a product measure.
\end{example}

\subsection{The main result}
Our main theorem is the following:
\begin{theorem}\label{thm:wp}
\begin{enumerate}
\item[(a)]
For every $\delta > 0$ and $\epsilon > 0$, there is a $\tau > 0$
such that for every weighted plurality function $f$ with weights $w_i$
and every probability
distribution $P$ on $[k]^n$, if $e_i(f, P) \le \tau$ and
there is a set $A \subset [n]$ such that
$\sum_i w_i P(X_i = a) \ge \sum_i w_i P(X_i = b) + \delta$ for all $i \in [n]$,
all $a \in A$ and all $b \not \in A$,
then $P(f(X) \in A) \ge 1 - \epsilon$.

\item[(b)]
If $f$ is not a weighted plurality function then there
exists a probability distribution $P$ on $[k]^n$ such that
$P(X_i = 2) > P(X_i = 1)$ for all $i \in [n]$
but $P(f(X) = 1) = 1$ (and hence $e_i(f, P) = 0$ for all $i$).
\end{enumerate}
\end{theorem}

We remark that the Theorem is constructive in the sense
that we can give an algorithm (based on solving a linear program)
which either constructs some weights $w_i$ witnessing the fact
that $f$ is a weighted plurality, or a probability distribution $P$
satisfying part (b).

Parts (a) and (b) of Theorem~\ref{thm:wp} are converse to one another
in the following sense: under the hypothesis of small effects,
part (a) says that if there is a gap between the popularity of
the most popular alternatives $A$ and the less popular alternatives
$A^c$ then a weighted plurality function will choose an alternative
in $A$. Part (b) shows that this property fails for every
function that is not a weighted plurality.
Note that part (a) has an important special case, which is closer to
the statement of~\cite{HKM}: if $P(X_i = a) \ge P(X_i = b) + \delta$
for all $i \in [n]$ and all $b \ne a$, then $f(X) = a$ with high
probability if the effects are small enough.

The remainder of the paper is devoted to the proof of Theorem~\ref{thm:wp}.

\begin{proof}[Proof of Theorem~\ref{thm:wp} (a)]
This part of the proof follows very closely the argument
in~\cite{HKM}.
Suppose that $f$ is a weighted plurality function
with weights $w_i$.
The first step is to show
that $f$ is ``correlated'' in some sense with each voter:
define $p_{ij} = P(X_i = j)$ and let $W_j$ be the (random) weight
assigned to alternative $j$: $W_j = \sum_{i: X_i = j}w_i$. Then
\begin{align}\label{eq:corr}
\notag & \E \sum_{i=1}^n w_i \sum_{j=1}^k \1{f(X) = j} (\1{X_i = j} - p_{ij}) \\
\notag &=
\E \left(\sum_{i,j} w_i \1{f(X) = j} \1{X_i = j}
- \sum_{i,j} \1{f(X) = j} w_i p_{ij}\right) \\
&=
\E \sum_{i,j} w_i \1{f(X) = j} \1{X_i = j}
- \sum_j P(f = j) \E W_j.
\end{align}
Now, let $\alpha_j = P(f = j)$ and set $\tilde \alpha_j = \alpha_j / (\sum_{i \in A} \alpha_i)$
for $j \in A$ and $\tilde \alpha_j = 0$ otherwise.
The first term of~\eqref{eq:corr} is just
\begin{multline}\label{eq:first}
\E \sum_{i,j} w_i \1{f(X) = j} \1{X_i = j}
=
\E \sum_j \1{f(X) = j} W_j \\
\ge
\E \sum_j \1{f(X) = j} \sum_i \tilde \alpha_i W_i
= \sum_i \tilde \alpha_i \E W_j
\end{multline}
since the winning alternative always has at least as much weight
as any convex combination of alternatives.
Since $\min_{j \in A} \E W_j \ge \max_{j \not \in A} \E W_j + \delta$,
we can plug~\eqref{eq:first} into~\eqref{eq:corr} to obtain
\begin{align*}
\eqref{eq:corr}
&\ge \sum_j \tilde \alpha_j \E W_j - \sum_j \alpha_j \E W_j \notag \\
&\ge \sum_{j\in A} (\tilde \alpha_j - \alpha_j) \delta \notag \\
&= \delta P(f \not \in A).
\end{align*}

Recalling that $e_i(f, P) \ge 4 \sum_j \Cov(\1{f=j}, \1{X_i=j})$, we have
\begin{align*}
\delta P(f \not \in A)
&\le \E \sum_{i=1}^n w_i \sum_{j=1}^k \1{f(X) = j} (\1{X_i = j} - p_{ij}) \\
&\le \frac{1}{4} \sum_i w_i e_i(f, P) \\
&\le \frac{\tau}{4}
\end{align*}
and so one direction of the theorem is proved once we take $\tau$
small enough that $\epsilon \ge \tau / (4\delta)$.
\end{proof}

The proof of the second part of the theorem follows the
idea of~\cite{HKM}, in that we use linear programming duality
to find a witness for $f$ being a weighted plurality function.
However, the details of the proof are quite different, since~\cite{HKM}
uses a well-known linear program (the fractional vertex cover of a
hypergraph) which does not extend beyond $k=2$.

The proof idea is this:
we will write down a linear program
and its dual. If the primal program has a large
enough value it will turn out that $f$ is a weighted plurality function.
Otherwise, the dual has a small value and
the dual variables witness the claim of Theorem~\ref{thm:wp} (b).
In particular, note that this proof provides the algorithm
that we mentioned after the statement of Theorem~\ref{thm:wp}.

First we make a trivial observation that will simplify
our linear program considerably:
if a function is neutral, it is easier to check whether it is
a weighted plurality because it is not necessary to try all possible
combinations of $a, b \in [k]$:
\begin{proposition}\label{prop:fair}
Suppose $f: [k]^n \to [k]$ is neutral. Then $f$ is a weighted
plurality if and only if there exist weights $w_1, \dots, w_n \in \R$
such that $f(x) = 1$ implies that
\[
\sum_{i: x_i = 1} w_i \ge \sum_{i: x_i = 2} w_i.
\]
\end{proposition}

We can write a linear program for checking whether a given neutral
function $f$ is a weighted plurality. The variables for this program
are $t$; $w_i$ for each $i \in [n]$ and $g_x$ for each $x \in [k]^n$
for which $f(x) = 1$. In standard form, the primal program is
the following:
\begin{alignat*}{2}
  \text{maximize }   & t_+ - t_- \\
  \text{subject to } & g_x \ge 0 \text{ for all }
                       x \in [k]^n\text{ such that } f(x) = 1 \\
                     & w_i \ge 0 \text{ for all } i \in [n] \\
                     & t_+ \ge 0 \text{ and } t_- \ge 0 \\
                     & \sum_i w_i = 1 \\
                     & \sum_{i: x_i = 1} w_i - \sum_{i: x_i = 2} w_i - g_x - (t_+ - t_-) = 0
                       \text{ for all } x \in [k]^n
                       \text{ with } f(x) = 1.
\end{alignat*}

\begin{proposition}
Let $t^*$ be the value of the above linear program.
If $t^* \ge 0$ then $f$ is a weighted plurality function.
\end{proposition}

\begin{proof}
Let $w_i$, $g_x$, $t_+$ and $t_-$
be feasible points such that $t_+ - t_- \ge 0$. Then,
for all $x$ with $f(x) = 1$,
\[
\sum_{i: x_i = 1} w_i - \sum_{i: x_i = 2} w_i = g_x + (t_+ - t_-) \ge 0
\]
and so $f$ satisfies the conditions of Proposition~\ref{prop:fair}.
\end{proof}

Now consider the dual program; since the primal is in
standard form, the dual is easy to write down. Let the
dual variables be $a$ and $q_x$ for all $x$ such that
$f(x) = 1$. Then the dual program is:
\begin{alignat*}{2}
  \text{minimize } & a_+ - a_- \\
  \text{subject to } & \sum_{x:f(x) = 1} q_x \le -1 \\
                     & \sum_{x:f(x) = 1} (\1{x_i = 1} - \1{x_i = 2}) q_x + (a_+ - a_-) \ge 0
                       \text{ for all } i \in [n] \\
                     & q_x \le 0 \text{ for all } x
                       \text{ such that } f(x) = 1 \\
                     & a_+ \le 0 \text{ and } a_- \le 0.
\end{alignat*}

\begin{proposition}
Let $a^*$ be the value of the above dual program.
If $a^* < 0$ then there exists a probability distribution on $[k]^n$
such that $P(X_i = 2) > P(X_i = 1)$ for all $i$ but $f(X) = 1$
almost surely.
\end{proposition}

\begin{proof}
Choose a feasible point with $a_+ - a_- < 0$ and define $p_x = -q_x / (\sum_x q_x)$.
Then $p_x \ge 0$ and $\sum_x p_x = 1$, so we can define a probability
distribution by $P(X = x) = p_x$ when $f(x) = 1$ and $P(X = x) = 0$
otherwise. Under this distribution, $f(X) = 1$ with probability 1.
On the other hand, with $a_+ - a_- < 0$
the constraints of the dual program imply that
\[
\sum_{x:f(x) = 1} \1{x_i = 1} q_x > \sum_{x:f(x) = 1} \1{x_i = 2} q_x
\]
for all $i$.
Thus,
\[
P(X_i = 1) = \sum_{x:f(x) = 1} \1{x_i = 1}p_x < \sum_{x:f(x) = 1} \1{x_i = 2}p_x = P(X_i = 2)
\]
for all $i$.
\end{proof}

To conclude the proof of Theorem~\ref{thm:wp}, note that both
the primal and dual programs are feasible and bounded
and so $a^* = t^*$.

\bibliography{plurality}
\bibliographystyle{plain}

\end{document}